\documentclass[reqno]{amsart}

\usepackage{amsmath,amsthm,amssymb,amsfonts,enumerate,graphicx,color,pstricks}

\usepackage[utf8]{inputenc}

\numberwithin{equation}{section} 
\theoremstyle{plain}
\newtheorem{theo+}           {Theorem}      [section]
\newtheorem{prop+}  [theo+]  {Proposition}
\newtheorem{coro+}  [theo+]  {Corollary}
\newtheorem{lemm+}  [theo+]  {Lemma}
\newtheorem{defi+}  [theo+]  {Definition}
\newtheorem{conj+}  [theo+]  {Conjecture}

\theoremstyle{definition}
\newtheorem{rema+}  [theo+]  {Remark}
\newtheorem{prob+}  [theo+]  {Problem}
\newtheorem{exam+}  [theo+]  {Example}

\newenvironment{theorem}{\begin{theo+}}{\end{theo+}}
\newenvironment{proposition}{\begin{prop+}}{\end{prop+}}

\newenvironment{lemma}{\begin{lemm+}}{\end{lemm+}}

\newcommand{\om}{\omega}

\newcommand{\ti}{\mathrm i}

\begin{document}

\baselineskip 18pt
\larger[2]
\title
[Proofs of some partition identities conjectured by Kanade and Russell] 
{Proofs of some partition identities\\ conjectured by Kanade and Russell}
\author{Hjalmar Rosengren}
\address
{Department of Mathematical Sciences
\\ Chalmers University of Technology and University of Gothenburg\\SE-412~96 G\"oteborg, Sweden}
\email{hjalmar@chalmers.se}
\urladdr{http://www.math.chalmers.se/{\textasciitilde}hjalmar}

\thanks{Supported by the Swedish Science Research
Council (Vetenskapsr\aa det). }

\begin{abstract}
Kanade and Russell conjectured several Rogers--Ramanujan-type partition identities, some of which are related to level $2$ characters of the affine Lie algebra $A_9^{(2)}$. Many of these conjectures
have been proved by Bringmann, Jennings--Shaffer and Mahlburg. We give new
proofs of five conjectures first proved by those authors, as well as four others that have been open until now. Our proofs for the new cases use
quadratic transformations for Askey--Wilson and Rogers polynomials. We also obtain some related results, including a new proof of a partition identity conjectured by Capparelli and first proved by Andrews.
\end{abstract}

\maketitle

\section{Introduction}  

The Rogers--Ramanujan identities  are
\begin{equation}\label{rr} \sum_{n=0}^\infty\frac{q^{n^2}}{(q;q)_n}=\frac 1{(q,q^4;q^5)_\infty},\qquad
\sum_{n=0}^\infty\frac{q^{n(n+1)}}{(q;q)_n}=\frac 1{(q^2,q^3;q^5)_\infty}.
\end{equation}
Here,  we use the standard notation
$$ (a;q)_n=\prod_{j=0}^{n-1}(1-aq^j)$$
and
$$(a_1,\dots,a_m;q)_n=(a_1;q)_m\dotsm (a_m;q)_n,$$
where $n$ may be infinite. These intriguing identities were first found by Rogers \cite{r2} and rediscovered by Ramanujan \cite{ra}. 

The identities \eqref{rr} have a clear combinatorial meaning; for instance, the first identity states that the number of partitions of $n$ without repeated or consecutive parts equals the number of partitions of $n$
into parts congruent to $\pm 1\ \operatorname{mod}\ 5$. They also have deep connections to affine Lie algebras.  Lepowsky and Milne \cite{lm} found that the product sides of \eqref{rr} can be interpreted as specialized characters of
 level 3 modules of the Lie algebra $A_1^{(1)}$. 
 Lepowsky and Wilson \cite{lw1,lw2} were able to use this fact to give Lie-theoretic proofs of \eqref{rr}.

There is now a huge literature on relations between Lie algebras and Rogers--Ramanujan-type identities. 
Of particular relevance to the present work is Misra's observation \cite{m} that the product sides in \eqref{rr} may alternatively be interpreted 
as level $2$ characters of $A_7^{(2)}$.
This motivated Kanade and Russell \cite{kr} to search for  Rogers--Ramanujan-type identities related to level  $2$ characters of 
$A_9^{(2)}$. 
They discovered the three conjectured identities \eqref{kcc}--\eqref{kce} below, which
 give explicit values of the function
\begin{equation}\label{f}F(u,v,w)=\sum_{i,j,k=0}^\infty\frac{(-1)^kq^{3k(k-1)+(i+2j+3k)(i+2j+3k-1)}u^iv^jw^k}{(q;q)_i(q^4;q^4)_j(q^6;q^6)_k}.\end{equation}
By an extended search, they found six more cases when $F$ appears to have a simple factorization. The resulting nine conjectures are
\begin{subequations}\label{krai}
\begin{align}
\label{kca}F(q,1,q^3)&=\frac{(q^3;q^{12})_\infty}{(q,q^2;q^4)_\infty},\\
\label{kcb}F(q^2,q^4,q^9)&=\frac{(q^9;q^{12})_\infty}{(q^2,q^3;q^4)_\infty},\\
\label{kcc}F(q^4,q^6,q^{15})&=\frac{1}{(q^4,q^5,q^6,q^7,q^8;q^{12})_\infty},\\
\label{kcd}F(q,q^6,q^9)&=\frac{1}{(q,q^4,q^6,q^8,q^{11};q^{12})_\infty},\\
\label{kce}F(q^2,q^2,q^9)&=\frac{(q^6;q^{12})_\infty}{(q^2,q^3,q^4;q^6)_\infty},\\
\label{fsa}F(q^3,q^5,q^{12})&=\frac 1{(q^3;q^4)_\infty(q^4,q^{5};q^{12})_\infty},\\
\label{fsc}F(q,q^3,q^6)&=\frac 1{(q;q^4)_\infty(q^4,q^{11};q^{12})_\infty},\\
\label{fsb}F(q,q,q^6)&=\frac 1{(q^3;q^4)_\infty(q,q^8;q^{12})_\infty},\\
\label{krx}F(q^2,q^{-1},q^6)&=\frac 1{(q;q^4)_\infty(q^7,q^8;q^{12})_\infty};
\end{align}
\end{subequations}
see Conjecture 5, 5a, 3, 1, 2, 6a, 4, 6 and 4a in \cite{kr}, respectively.
We have stated them more or less in order of increasing difficulty, at least in our approach.
For each conjecture, Kanade and Russell gave a combinatorial interpretation in terms of partitions.
We stress that \cite{kr} contains several other conjectures that will not be discussed here.

The identities \eqref{kca}--\eqref{kce} were proved by Bringmann et al.\ \cite{bjm}. In the present work, we give a more streamlined proof of those results, and prove the remaining conjectures \eqref{fsa}--\eqref{krx} for the first time.
 In each case, we use an integral representation of the triple sum to reduce it to a single sum. The identities \eqref{kca}--\eqref{kcb} then follow from classical $q$-hypergeometric summation formulas. The identities \eqref{kcc}--\eqref{kce} are also reduced to classical summations, but we need Watson's quintuple product identity \cite{w}
 to simplify the result. The identities \eqref{fsa}--\eqref{krx} are  harder since we need one-variable summations that cannot be found in the literature.  
 To prove them we use  $q$-orthogonal polynomials.
 Roughly speaking, we identify the series we wish to sum with generating functions for Askey--Wilson polynomials and then use quadratic transformations,
 relating Askey--Wilson and Rogers polynomials, to compute them. 
 It is worth mentioning that Rogers polynomials first appeared in Rogers' original proof of \eqref{rr}. 

The plan of the paper is as follows. After the preliminary \S \ref{ps}, we
obtain in \S \ref{rtss} five one-parameter families of reduction formulas, which reduce the triple series $F$ to single series. The left-hand sides of \eqref{krai} all belong to one of these families.
In \S \ref{bjms} we briefly describe how the identities \eqref{kca}--\eqref{kce} are obtained in our approach. In \S \ref{omss} we give some results for other types of multiple series that can be obtained with our methods. In particular, we obtain a new proof of a partition identity related to level $3$ characters of $A_2^{(2)}$,
which was
conjectured by Capparelli \cite{c} and first proved by  Andrews \cite{a}.   The most technical part of the paper is \S \ref{ncs}, where we prove \eqref{fsa}--\eqref{krx}, as well as a
related identity given in Theorem \ref{ntst}.

{\bf Acknowledgments:} This work was supported by the Swedish Research Council. I thank Chris Jennings-Shaffer for his useful comments on the manuscript, leading to many improvements.

\section{Preliminaries}\label{ps}
Throughout, $\omega$ is a primitive cubic root of unity and $q$ a number with $0<|q|<1$. 

For ease of reference, 
we recall some classical results  for the basic hypergeometric series \cite{gr}
$$ {}_{r}\phi_s\left(\begin{matrix}
a_1,\dots,a_{r}\\b_1,\dots,b_s
\end{matrix};q,z\right)=\sum_{n=0}^\infty\frac{(a_1,\dots,a_{r};q)_n}{(b_1,\dots,b_s;q)_n}\,\left((-1)^nq^{\binom n2}\right)^{1+s-r}z^n.$$
 Namely, we have Euler's $q$-exponential identities
\begin{equation}\label{eqe}\sum_{n=0}^\infty \frac {z^n}{(q;q)_n}=\frac 1{(z;q)_\infty},\qquad
\sum_{n=0}^\infty \frac {q^{\binom n2}z^n}{(q;q)_n}={(-z;q)_\infty}, \end{equation}
the $q$-binomial theorem
\begin{equation}\label{qbt}{}_1\phi_0(a;q,z)=\sum_{n=0}^\infty \frac {(a;q)_n}{(q;q)_n}\,z^n=\frac{(az;q)_\infty}{(z;q)_\infty} \end{equation}
and its disguised version
\begin{equation}\label{aqbt}{}_2\phi_1\left(\begin{matrix}
a,-a\\-q
\end{matrix};q,z\right)=\sum_{n=0}^\infty\frac{(a^2;q^2)_n}{(q^2;q^2)_n}\,z^n=\frac{(a^2z;q^2)_\infty}{(z;q^2)_\infty}. \end{equation}
Applying one of Heine's transformation formulas to \eqref{aqbt} gives the Bailey--Daum summation
\begin{equation}\label{bds}{}_2\phi_1\left(\begin{matrix}
a,b\\aq/b\,
\end{matrix};q, -\frac{q}{b}\right)=\frac{(-q;q)_\infty(aq,aq^2/b^2;q^2)_\infty}{(aq/b,-q/b;q)_\infty}. 
\end{equation}
The $q$-Gauss summation is
\begin{equation}\label{qgs}{}_2\phi_1\left(\begin{matrix}
a,b\\c
\end{matrix};q, \frac{c}{ab}\right)=\frac{(c/a,c/b;q)_\infty}{(c,c/ab;q)_\infty}. \end{equation}
Finally, we mention the transformation
\begin{equation}\label{iht}{}_2\phi_1\left(\begin{matrix}
a,b\\c
\end{matrix};q,z\right)=\frac{(az,bz;q)_\infty}{(c,z;q)_\infty}
\,{}_2\phi_2\left(\begin{matrix}
z,abz/c\\az,bz
\end{matrix};q,c\right),
\end{equation}
which follows by combining \cite[(III.2) and (III.4)]{gr}.

We will need Jacobi's triple product identity
\begin{equation}\label{jtp}(q,z,q/z;q)_\infty=\sum_{n=-\infty}^\infty(-1)^nq^{\binom n2}z^n\end{equation}
as well as  Watson's quintuple product identity in the form \cite[Lemma~6.2]{r}
$$\frac 1{1+z}\frac{(z^2,z^{-2};q^3)_\infty}{(z,z^{-1};q^3)_\infty}=\frac{1-\om}{3}\frac{(q;q)_\infty}{(q^3;q^3)_\infty}\big((zq\om,\omega^2/z;q)_\infty-\om^2(zq\om^2,\om/z;q)_\infty\big). $$
If we replace $q$ by $q^4$ and then let $z=q^{-5}$, $z=q^{-1}$ and $z=q^{-3}$, respectively, we obtain
\begin{subequations}\label{wqs}
\begin{align}
\label{wqa}\frac{(q^2,q^{10};q^{12})_\infty}{(q^4,q^5,q^7,q^8;q^{12})_\infty}&=\frac{\om-\om^2}{3q}
\big((q\om,q^3\om^2;q^4)_\infty-(q\om^2,q^3\om;q^4)_\infty\big),\\
\label{wqb}\frac{(q^2,q^{10};q^{12})_\infty}{(q,q^4,q^8,q^{11};q^{12})_\infty}&=\frac{1-\om^2}{3}
\big((q\om,q^3\om^2;q^4)_\infty-\om(q\om^2,q^3\om;q^4)_\infty\big),\\
\label{wqc}\frac{(q^6;q^{12})_\infty^2}{(q^3,q^4,q^8,q^9;q^{12})_\infty}&=\frac{1-\om}{3}
\big((q\om,q^3\om^2;q^4)_\infty-\om^2(q\om^2,q^3\om;q^4)_\infty\big).
 \end{align}
\end{subequations}

\section{Reduction from triple to single series}
\label{rtss}

In \cite{bjm}, Bringmann et al.\ showed that the nine
conjectures \eqref{krai} can be reduced to
sums of single series, which they could evaluate in the first five cases.
In this section we will obtain more general reductions using the following integral representation of the function \eqref{f}. 

\begin{proposition}\label{fip}
We have
\begin{equation}\label{fir}F(u,v,w)
=(q^2;q^2)_\infty\oint\frac{(1/z,q^2z;q^2)_\infty(-wz^3;q^6)_\infty}{(-uz;q)_\infty(vz^2;q^4)_\infty}\frac{dz}{2\pi\ti z}, 
\end{equation}
where the integration is over a positively oriented contour separating $0$ from all poles of the integrand.
\end{proposition}

\begin{proof}
We may take the contour close enough to zero so that 
$|uz|<1$ and $|vz^2|<1$. By \eqref{eqe} and \eqref{jtp}, the 
right-hand side of \eqref{fir} equals
$$\oint\sum_{i=0}^\infty\frac{(-uz)^i}{(q;q)_i}\sum_{j=0}^\infty\frac{(vz^2)^j}{(q^4;q^4)_j}\sum_{k=0}^\infty\frac{q^{3k(k-1)}(wz^3)^k}{(q^6;q^6)_k}\sum_{l=-\infty}^{\infty}
q^{l(l-1)}(-z^{-1})^l\,\frac{dz}{2\pi\ti z}. $$
The integration picks out the constant term in this Laurent series, that is, the term with $l=i+2j+3k$.
\end{proof}

We will also need the following integral representation of the function
 ${}_2\phi_1$. Although it is an easy consequence of known results, we have not found it in the literature.

\begin{proposition}\label{pirp}
For $|t|<1$, we have the contour integral representation
$${}_2\phi_1\left(\begin{matrix}
a,b\\c
\end{matrix};q, t\right)
=\frac{(q;q)_\infty}{(c,t;q)_\infty}\oint\frac{(abz,cz,qz/t,t/z;q)_\infty}{(az,bz,cz/t;q)_\infty}\,\frac{dz}{2\pi\ti z},
 $$
 where the integral is over a positively oriented contour separating $0$ from all poles of the integrand.
\end{proposition} 

Note that the right-hand side displays the $\mathrm S_3\times\mathrm S_2$-symmetry of the
function ${}_2\phi_1$ implied by Heine's transformations \cite[(III.1)--(III.3)]{gr}.
Moreover, it gives the  analytic continuation of the left-hand side
to $t\in\mathbb C\setminus q^{\mathbb Z_{\leq 0}}$. 

\begin{proof}
We start from the more general integral
$$I=\oint\frac{(abz,cz,qz/t,t/z;q)_\infty}{(az,bz,cz/t,d/z;q)_\infty}\,\frac{dz}{2\pi\ti z},$$
which has an additional sequence of poles at $(dq^k)_{k=0}^\infty$. We take $d$ close to zero so that all those poles are inside the contour.  
 If $|t|<|d|$, then by \cite[Eq.\ (4.10.8)]{gr} the 
integral  equals the sum of the residues inside the contour, which can be written
$$\frac{(abd,cd,dq/t,t/d;q)_\infty}{(q,ad,bd,cd/t;q)_\infty}\,{}_3\phi_2\left(\begin{matrix}
ad,bd,cd/t\\abd,cd
\end{matrix};q,\frac td\right). $$
Applying the transformation formula \cite[Eq.\ (III.9)]{gr} gives
$$I=\frac{(abd,dq/t,t,c;q)_\infty}{(q,ad,bd,cd/t;q)_\infty}\,{}_3\phi_2\left(\begin{matrix}
a,b,cd/t\\c,abd
\end{matrix};q,t\right). $$
By analytic continuation, this holds for $|t|<1$. We now obtain the desired result by letting $d\rightarrow 0$.
\end{proof}

Replacing  $z$ by $tz$ and relabelling the parameters gives the equivalent identity
 \begin{multline}\label{tpia} \oint\frac{(\alpha_1z,\alpha_2z,qz,1/z;q)_\infty}{(\beta_1z,\beta_2z,\beta_3z;q)_\infty}\,\frac{dz}{2\pi\ti z}\\
 =\frac{(\beta_1,\alpha_1/\beta_1;q)_\infty}{(q;q)_\infty}\,{}_2\phi_1\left(\begin{matrix}\alpha_2/\beta_2,\alpha_2/\beta_3\\ \beta_1\end{matrix};q, \frac{\alpha_1}{\beta_1}\right),\qquad \alpha_1\alpha_2=\beta_1\beta_2\beta_3.
 \end{multline}
This can be compared with \eqref{fir}, where the integrand can be written
$$\frac{(-w^{1/3}z,-w^{1/3}\om z,-w^{1/3}\om^2z,q^2z,1/z;q^2)_\infty}{(-uz,-uqz,v^{1/2}z,-v^{1/2}z;q^2)_\infty}.
$$
We obtain an integral of the form \eqref{tpia} (with $q$ replaced by $q^2$) if 
$w^{1/3}\in\{u,uq,v^{1/2}\}$ and in addition $w=u^2vq$. That is, there are three one-parameter cases where  $F$ can be reduced to the function ${}_2\phi_1$.

\begin{proposition}\label{rip}
We have the reduction identities
\begin{subequations}\label{ripi}
\begin{align}
 \nonumber F(u,uq^{-1},u^3)&=(-u^{1/2}q^{-1/2},u^{1/2}q^{1/2}\omega^2;q^2)_\infty\\
\label{fra}&\quad\times\,{}_2\phi_1\left(\begin{matrix}
q^{-1}\omega,-u^{1/2}q^{1/2}\omega\\-u^{1/2}q^{-1/2}
\end{matrix};q^2,u^{1/2}q^{1/2}\omega^2 \right),
\\
F(u,uq^2,u^3q^3)&=
\label{frc}(-u,q\omega^2;q^2)_\infty\,{}_2\phi_1\left(\begin{matrix}
u^{1/2}\omega,-u^{1/2}\omega \\-u
\end{matrix};q^2,q\omega^2 \right),\\
F(u,u^4q^2,u^6q^3)&=
\label{fre}(-uq,u\omega^2;q^2)_\infty\,{}_2\phi_1\left(\begin{matrix}
-\omega,uq\omega \\-uq
\end{matrix};q^2,u\omega^2 \right),
\end{align}
\end{subequations}
assuming that the ${}_2\phi_1$-series converge.
\end{proposition}

Applying
\eqref{iht} to \eqref{tpia} gives the equivalent identity
 \begin{multline}\label{tpic} \oint\frac{(\alpha_1z,\alpha_2z,qz,1/z;q)_\infty}{(\beta_1z,\beta_2z,\beta_3z;q)_\infty}\,\frac{dz}{2\pi\ti z}\\
 =\frac{(\beta_2,\beta_3;q)_\infty}{(q;q)_\infty}\,{}_2\phi_2\left(\begin{matrix}\alpha_1/\beta_1,\alpha_2/\beta_1\\ \beta_2,\beta_3\end{matrix};q, \beta_1\right),\qquad \alpha_1\alpha_2=\beta_1\beta_2\beta_3.
 \end{multline}
This leads to
 alternative representations for \eqref{ripi} as ${}_2\phi_2$-series, which have infinite radius of convergence and hence hold for all $u$.

The Kanade--Russell conjectures \eqref{kca}--\eqref{kcd}
and \eqref{fsa}--\eqref{fsc} are all reduced to ${}_2\phi_1$-evaluations
by Proposition \ref{rip}. In the remaining three conjectures, \eqref{kce} and \eqref{fsb}--\eqref{krx},
the variables in $F(u,v,w)$ satisfy $w=u^2vq^3$ rather than $w=u^2vq$.
They can be expressed as integrals of the form 
\eqref{tpia} with $\alpha_1\alpha_2=\beta_1\beta_2\beta_3q$. It is easy to see that any such integral can be expressed (in several different ways) as the sum of two integrals 
with $\alpha_1\alpha_2=\beta_1\beta_2\beta_3$, and hence as a sum of two ${}_2\phi_1$-series. If  $\beta_3=-\beta_2$,  there is a more elegant  expression as a  ${}_2\phi_2$-series.

\begin{lemma}\label{crl}
When  $\alpha_1\alpha_2=-\beta_1\beta_2^2q$, then
$$\oint\frac{(\alpha_1z,\alpha_2z,qz,1/z;q)_\infty}{(\beta_1z,\beta_2z,-\beta_2z;q)_\infty}\,\frac{dz}{2\pi\ti z}
 =\frac{(\beta_2^2q^2;q^2)_\infty}{(q;q)_\infty}\,{}_2\phi_2\left(\begin{matrix}
 \alpha_1/\beta_1,\alpha_2/\beta_1
 \\ \beta_2q,-\beta_2q\end{matrix};q, \beta_1\right).$$
\end{lemma}

\begin{proof}
Let $I$ denote the given integral.
Inserting the factor
$$1=\frac{(1+\beta_2z)+(1-\beta_2z)}{2}, $$
it splits as a sum, where we may evaluate each term using \eqref{tpic}.
This gives
\begin{align*}I&=\frac{1}{2}\oint\frac{(\alpha_1z,\alpha_2z,qz,1/z;q)_\infty}{(\beta_1z,\beta_2z,-\beta_2qz;q)_\infty}\,\frac{dz}{2\pi\ti z}
+\frac{1}{2}\oint\frac{(\alpha_1z,\alpha_2z,qz,1/z;q)_\infty}{(\beta_1z,\beta_2qz,-\beta_2z;q)_\infty}\,\frac{dz}{2\pi\ti z}\\
&=\frac{1}{2(q;q)_\infty}
\left((\beta_2,-\beta_2q;q)_\infty\,{}_2\phi_2\left(\begin{matrix}\alpha_1/\beta_1,\alpha_2/\beta_1\\ \beta_2,-\beta_2q\end{matrix};q, {\beta_1}\right)\right.\\
&\quad\left.+
(\beta_2q,-\beta_2;q)_\infty\,{}_2\phi_2\left(\begin{matrix}\alpha_1/\beta_1,\alpha_2/\beta_1\\ \beta_2q,-\beta_2\end{matrix};q, {\beta_1}\right).
\right)
\end{align*}
Adding the series termwise gives
$$I=\frac{(\beta_2^2q^2;q^2)_\infty}{(q;q)_\infty}
\sum_{k=0}^\infty\frac{(\alpha_1/\beta_1,\alpha_2/\beta_1;q)_k}{(\beta_2q,-\beta_2q;q)_k}\,(-1)^kq^{\binom k2}\beta_1^k\frac{(1-\beta_2q^k)+(1+\beta_2q^k)}{2},
 $$
which simplifies to
  the desired expression.

\end{proof}

Combining Proposition \ref{fip} and Lemma \ref{crl} gives the following reduction identities. Note that \eqref{kce} and \eqref{fsb} are series of the form \eqref{f32} whereas
\eqref{krx} is of the form \eqref{far}.

\begin{proposition}\label{crp}
We have 
\begin{subequations}\label{crpi}
\begin{align}
\label{f32}F(u,u,u^3q^3)&=
(uq^4;q^4)_\infty\,{}_2\phi_2\left(\begin{matrix}
q\omega,q\om^2\\u^{1/2}q^{2},-u^{1/2}q^{2}
\end{matrix};q^2,-u \right),\\
\label{far}F(u,uq^{-3},u^3)&=(uq;q^4)_\infty\,{}_2\phi_2\left(\begin{matrix}
q^{-1}\omega,q^{-1}\om^2\\u^{1/2}q^{1/2},-u^{1/2}q^{1/2}
\end{matrix};q^2,-uq \right).
\end{align}
\end{subequations}
\end{proposition}

\section{Results of Bringmann, Jennings-Shaffer and Mahlburg}
\label{bjms}

For completeness, we show how to prove the conjectures
 \eqref{kca}--\eqref{kce} in our approach. 
 Bringmann et al.\ \cite{bjm} also obtain these results 
 by reducing triple summations to single summations, which are either the same as ours or easily seen to be equivalent. The main virtue of our approach is the more streamlined way to obtain such reductions as explained in \S \ref{rtss}.
 
 \begin{proof}[Proof of \eqref{kca}]
If we let $u=q$ in \eqref{fra} we get
$$F(q,1,q^3) =(-1,q\omega^2;q^2)_\infty\,{}_2\phi_1\left(\begin{matrix}
q^{-1}\omega,-q\omega\\-1
\end{matrix};q^2,q\omega^2 \right)=(-q^2;q^2)_\infty(q\om,q\om^2;q^4)_\infty,$$
where we used the Bailey--Daum summation \eqref{bds}. This can be rewritten as \eqref{kca}.
\end{proof}

\begin{proof}[Proof of \eqref{kcb}]
If we let  $u=q^2$ in \eqref{frc}, we get
\begin{align*}F(q^2,q^4,q^9)&=(-q^2,q\om^2;q^2)_\infty\,{}_2\phi_1\left(\begin{matrix}
q\omega,-q\omega\\-q^2
\end{matrix};q^2,q\omega^2 \right)\\
&
=(-q^2;q^2)_\infty(q^3\om,q^3\om^2;q^4)_\infty, \end{align*}
by either \eqref{aqbt} or  \eqref{bds}. This can be written as \eqref{kcb}.
\end{proof}

\begin{proof}[Proof of \eqref{kcc}]
If we let $u=q^4$ in \eqref{frc}, we get
$$F(q^4,q^6,q^{15})=(-q^4,q\om^2;q^2)_\infty\,{}_2\phi_1\left(\begin{matrix}
q^2\omega,-q^2\omega\\-q^4
\end{matrix};q^2,q\omega^2 \right). $$
This ${}_2\phi_1$ is contiguous to a ${}_1\phi_0$ in base $q^4$. More precisely, the terms can be expressed as
$$\frac{(q^2\om,-q^2\om;q^2)_k}{(q^2,-q^4;q^2)_k}(q\om^2)^k
=\frac{\om(1+q^2)}{(1-\om^2)q}\frac{(\om^2;q^4)_{k+1}}{(q^4;q^4)_{k+1}}(q\om^2)^k(1-q^{2k}).
 $$
 Replacing $k$ by $k-1$ we find that
\begin{align*}F(q^4,q^6,q^{15})&=\frac{\om(-q^2,q\om^2;q^2)_\infty}{(1-\om^2) q}\Big({}_1\phi_0\left(
\omega^2;q^4,q\omega^2 \right)-\,{}_1\phi_0\left(
\omega^2;q^4,q^3\omega^2 \right)\Big)\\
&=\frac{\om(-q^2;q^2)_\infty}{(1-\om^2) q}\Big((q\om,q^3\om^2;q^4)_\infty
-(q\om^2,q^3\om;q^4)_\infty
\Big),
\end{align*}
where we used the $q$-binomial theorem \eqref{qbt}.
Applying \eqref{wqa} and simplifying we get \eqref{kcc}.
\end{proof}

\begin{proof}[Proof of \eqref{kcd}]
If we let $u=q$ in \eqref{fre} we get
$$ F(q,q^6,q^{9})=(-q^2,q\om^2;q^2)_\infty\,{}_2\phi_1\left(\begin{matrix}
q^2\omega,-\omega\\-q^2
\end{matrix};q^2,q\omega^2 \right). $$
This time we write the terms as
$$\frac{(q^2\om,-\om;q^2)_k}{(q^2,-q^2;q^2)_k}(q\om^2)^k
=\frac{1}{1-\om}\frac{(\om^2;q^4)_{k}}{(q^4;q^4)_{k}}
(q\om^2)^{k}(1-\om q^{2k}), $$
which gives
\begin{align*}F(q,q^6,q^{9})&=
\frac{(-q^2,q\om^2;q^2)_\infty}{1-\om}\Big({}_1\phi_0\left(
\omega^2;q^4,q\omega^2 \right)-\om\,{}_1\phi_0\left(
\omega^2;q^4,q^3\omega^2 \right)\Big)\\
 &=\frac{(-q^2;q^2)_\infty}{(1-\om)}\Big((q\om,q^3\om^2;q^4)_\infty
-\om(q\om^2,q^3\om;q^4)_\infty
\Big).
 \end{align*}
 Applying \eqref{wqb} we arrive at \eqref{kcd}. 
 \end{proof}
 
 \begin{proof}[Proof of \eqref{kce}]
 Although $F(q^2,q^2,q^9)$ is a special case of \eqref{f32}, we do not use that identity. Instead, we 
 start from
 \eqref{fir}, which gives
 $$F(q^2,q^2,q^9)=(q^2;q^2)_\infty\oint\frac{(-q^3\om z,-q^3\om^2 z,1/z,q^2z;q^2)_\infty}{(qz,-qz,-q^2z;q^2)_\infty}\frac{dz}{2\pi\ti z}. $$
 Inserting the factor
 $$1=\frac{(1+q\om z)-\om^2(1+q\om^2z)}{1-\om^2} $$
 gives
 \begin{align*}F(q^2,q^2,q^9)&=\frac{(q^2;q^2)_\infty}{1-\om^2}\left(\oint\frac{(-q\om z,-q^3\om^2z,1/z,q^2z;q^2)_\infty}{(qz,-qz,-q^2z;q^2)_\infty}\frac{dz}{2\pi\ti z}\right.\\
&\qquad \left. -\om^2
 \oint\frac{(-q^3 \om z,-q\om^2 z,1/z,q^2z;q^2)_\infty}{(qz,-qz,-q^2z;q^2)_\infty}\frac{dz}{2\pi\ti z}\right)\\
& =\frac{(-q^2;q^2)_\infty}{1-\om^2}\left((q\om^2;q^2)_\infty\,
 {}_2\phi_1\left(\begin{matrix}
\omega,-\om\\-q^{2}
\end{matrix};q^2,q\om^2 \right)\right.\\
&\qquad\left.-\om^2(q\om;q^2)_\infty\,
 {}_2\phi_1\left(\begin{matrix}
\omega^2,-\om^2\\-q^{2}
\end{matrix};q^2,q\om \right)
 \right)\\
 &=\frac{(-q^2;q^2)_\infty}{1-\om^2}\left((q\om,q^3\om^2;q^4)_\infty
 -\om^2(q\om^2,q^3\om;q^4)_\infty
 \right),
\end{align*}
 where we used first \eqref{tpia} and then \eqref{aqbt}. Applying
\eqref{wqc} we arrive at \eqref{kce}. 
\end{proof}

\section{Other types of multiple series}
\label{omss}

The integral \eqref{tpia} can be used to study other types of multiple series apart from $F$. We will give some examples that seem very close to the conjectures \eqref{krai}. We first observe that  \eqref{tpia} can be evaluated 
  by \eqref{aqbt} (or, alternatively, \eqref{bds})  whenever
   $\beta_1=-q$ and $\beta_2=-\beta_3$. 
    Indeed, assuming also $|\alpha_1|<q$   we have
\begin{multline}\label{bdi}  \oint\frac{(\alpha_1z,\alpha_2z,qz,1/z;q)_\infty}{(-qz,\beta_2z,-\beta_2z;q)_\infty}\,\frac{dz}{2\pi\ti z}\\
\begin{split}& =\frac{(-q,-\alpha_1q^{-1};q)_\infty}{(q;q)_\infty}\,{}_2\phi_1\left(\begin{matrix}\alpha_2/\beta_2,-\alpha_2/\beta_2\\ -q\end{matrix};q, -\alpha_1q^{-1}\right)\\
 &
 =\frac{(-q;q)_\infty(-\alpha_1,-\alpha_2;q^2)_\infty}{(q;q)_\infty},\qquad \alpha_1\alpha_2=\beta_2^2q.\end{split}
 \end{multline}
 By analytic continuation,
 the final expression  is valid  for general values of $\alpha_1$. 
 
 We will only consider \eqref{bdi} in the special case
 $\alpha_2=\omega\alpha_1$. Replacing $q$ by $q^2$, it can be written
\begin{equation}\label{tsi}\oint\frac{(-u^3q^3z^3;q^6)_\infty(q^2z,1/z;q^2)_\infty}{(-q^2z,uz,-uz,-uqz;q^2)_\infty}\,\frac{dz}{2\pi\ti z}
 =\frac{(-q^2;q^2)_\infty( u^3q^3;q^{12})_\infty}{(q^2;q^2)_\infty(uq;q^4)_\infty}. \end{equation}
 Let  us write
 $$f(u) =(-q^2z,uz,-uz,-uqz;q^2)_\infty$$
 for the denominator  on the left-hand side. 
 We may obtain non-trivial summations by combining factors from $f$ before expanding the integrand as a Laurent series.
 The most compact expression is obtained for $u=-q$, since
 $$f(-q)=(qz,-qz,q^2z,-q^2z;q^2)_\infty=(q^2z^2;q^2)_\infty. $$
 By \eqref{eqe} and \eqref{jtp}, the left-hand side of \eqref{tsi} then takes the form
 \begin{multline*} \oint\frac{(q^6z^3;q^6)_\infty(q^2z,1/z;q^2)_\infty}{(q^2z^2;q^2)_\infty}\,\frac{dz}{2\pi\ti z}\\
 \begin{split}&=\frac 1{(q^2;q^2)_\infty}\sum_{j,k=0}^\infty\sum_{l=-\infty}^\infty
 \frac{(-1)^{k+l}q^{2j+6\binom k2+6k+2\binom l2}}{(q^2;q^2)_j(q^6;q^6)_k}
 \oint z^{2j+3k-l}\frac{dz}{2\pi\ti z}\\
& =\frac 1{(q^2;q^2)_\infty}\sum_{j,k=0}^\infty
 \frac{q^{4j^2+12jk+12k^2}}{(q^2;q^2)_j(q^6;q^6)_k}.
\end{split} \end{multline*}
 After replacing $q^2$ by $q$ we obtain the following double summation.
 It was shown in \cite{kr} and \cite{k} that it is equivalent to a partition identity 
   conjectured by Capparelli \cite{c} and first proved by Andrews \cite{a}. 
   We now have a new proof of this result, based on Proposition \ref{pirp}.

 \begin{proposition}[Andrews, Kanade--Russell, Kur\c{s}ung\"oz]
 We have
$$ \sum_{j,k=0}^{\infty}\frac{q^{2j^2+6jk+6k^2}}{(q;q)_{j}(q^3;q^3)_k}=\frac{1}{(q^3;q^6)_\infty(q^{2},q^{10};q^{12})_\infty}.$$
 \end{proposition}

If we can write $f(u)$ as the product of two factors of the form $(az^k;\pm q^l)_\infty$, we obtain in the same way a triple summation. 
In \S \ref{bjms} we obtained  \eqref{kca} from the  factorization
$$ f(1)=(z,-z,-qz,-q^2z;q^2)_\infty=(-qz;q)_\infty(z^2;q^4)_\infty$$
and \eqref{kcb} from
$$f(q^2)=(q^2z,-q^2z,-q^2z,-q^3z;q^2)_\infty=(-q^2z;q)_\infty(q^4z^2;q^4)_\infty. $$
There are essentially 
three other  cases when this can be done, namely,
 \begin{align*}f(-q^3)&=(-q^2z,q^3z,-q^3z,q^4z)=
  (-q^2z;q)_\infty(q^3z;q)_\infty,\\
  f(q)&=(qz,-qz,-q^2z,-q^2z;q^2)_\infty=(-qz;q)_\infty(qz;-q)_\infty,\\
  f(q^3)&=(-q^2z,q^3z,-q^3z,-q^4z;q^2)_\infty=(-q^2z;q)_\infty(q^3z;-q)_\infty. 
  \end{align*}
  The corresponding three summations are as follows.
To our knowledge, they are new. 
We do not know whether they have natural interpretations as partition identities,
  nor whether they are related to Lie algebras.
 
 \begin{theorem}
 We have
 \begin{subequations}
 \begin{align}
 \label{tsf}
 \sum_{i,j,k=0}^{\infty}\frac{(-1)^{j}q^{2\binom{i+j+3k}{2}+6\binom k2+2i+3j+12k}}{(q;q)_i(q;q)_{j}(q^6;q^6)_k}&=\frac{(q^4,q^{20};q^{24})_\infty}{(q^2;q^{4})_\infty},\\
 \label{tsc}
 \sum_{i,j,k=0}^{\infty}\frac{(-1)^{j+k}q^{2\binom{i+j+3k}{2}+6\binom k2+i+j+6k}}{(q;q)_{i}(-q;-q)_j(q^6;q^6)_k}&=\frac{1}{(q^2;q^4)_\infty(q^2,q^{10};q^{12})_\infty}, \\
  \label{tse}
 \sum_{i,j,k=0}^{\infty}\frac{(-1)^{j+k}q^{2\binom{i+j+3k}{2}+6\binom k2+2i+3j+12k}}{(q;q)_i(-q;-q)_{j}(q^6;q^6)_k}&=\frac{1}{(q^2;q^{4})_\infty(q^4,q^8;q^{12})_\infty}.
 \end{align}
 \end{subequations}
 \end{theorem}
 
We have found one more result for the type of triple series 
appearing in \eqref{tsc}--\eqref{tse}. The proof, which is parallel  to that of
 \eqref{fsa}, is given in \S \ref{ncs}.

 \begin{theorem}\label{ntst}
 We have 
 \begin{multline}\label{ntss}\sum_{i,j,k=0}^{\infty}\frac{(-1)^{j+k}q^{2\binom{i+j+3k}{2}+6\binom k2+2i+\frac 32j+9k}}{(q;q)_i(-q;-q)_{j}(q^6;q^6)_k}\\
 =\frac{1}{(-q^{3/2},q^{5/2};-q^3)_\infty(q^4;q^6)_\infty(q^2;q^{12})_\infty}.
 \end{multline}
 \end{theorem}

\section{New cases of the Kanade--Russell conjectures}
 \label{ncs}

In this section we  prove the conjectures
 \eqref{fsa}--\eqref{krx} as well as   \eqref{ntss}.

\subsection{Reduction to one-variable series}

 We first reduce the results to one-variable summations.
By \eqref{frc} and \eqref{crpi} we have the reduction formulas
\begin{align*}
F(q^3,q^5,q^{12})&=(-q^3,q\omega^2;q^2)_\infty\,{}_2\phi_1\left(\begin{matrix}
q^{3/2}\omega,-q^{3/2}\omega \\-q^3
\end{matrix};q^2,q\omega^2 \right),\\
F(q,q^3,q^6)&=(-q,q\omega^2;q^2)_\infty\,{}_2\phi_1\left(\begin{matrix}
q^{1/2}\omega,-q^{1/2}\omega \\-q
\end{matrix};q^2,q\omega^2 \right),\\
F(q,q,q^6)&=
(q^5;q^4)_\infty\,{}_2\phi_2\left(\begin{matrix}
q\omega,q\om^2\\q^{5/2},-q^{5/2}
\end{matrix};q^2,-q \right),\\
F(q^2,q^{-1},q^6)&=(q^3;q^4)_\infty\,{}_2\phi_2\left(\begin{matrix}
q^{-1}\omega,q^{-1}\om^2\\q^{3/2},-q^{3/2}
\end{matrix};q^2,-q^3 \right).
\end{align*}
For the first two cases, Bringmann et al.\ \cite{bjm} found equivalent reductions
but for the last two their reductions are more complicated.
Replacing $q$ by $q^{1/2}$, it follows that the conjectures \eqref{fsa}--\eqref{krx} are equivalent to the following summations.

\begin{theorem}\label{krta}
The following summation formulas hold:
\begin{subequations}\label{ksba}
\begin{align}
 \label{ksa}\,{}_2\phi_1\left(\begin{matrix}
q^{3/4} \omega,-q^{3/4}\omega\\  -q^{3/2}
\end{matrix};q,{q^{1/2}}{\omega^2}\right)
&=\frac{(1+q^{1/2})(q^{1/2},q^{9/2};q^6)_\infty}{(q^{1/2}\omega^2;q)_\infty(q;q^2)_\infty(q^2;q^6)_\infty},\\
\label{ksb}\,{}_2\phi_1\left(\begin{matrix}
q^{1/4} \omega,-q^{1/4}\omega\\  -q^{1/2}
\end{matrix};q,{q^{1/2}}{\omega^2}\right)
&=\frac{(q^{3/2},q^{7/2};q^6)_\infty}{(q^{1/2}\omega^2;q)_\infty(q;q^2)_\infty(q^2;q^6)_\infty},\\
\label{ksc}{}_2\phi_2\left(\begin{matrix}
q^{1/2}\omega,q^{1/2}\om^2\\q^{5/4},-q^{5/4}
\end{matrix};q,-q^{1/2} \right)&=\frac {1-q^{1/2}}{(q^{1/2};q)_\infty(q^{1/2},q^4;q^{6})_\infty}
,\\
\label{ksd}{}_2\phi_2\left(\begin{matrix}
q^{-1/2}\omega,q^{-1/2}\om^2\\q^{3/4},-q^{3/4}
\end{matrix};q,-q^{3/2} \right)&=\frac 1{(q^{1/2};q)_\infty(q^{7/2},q^4;q^{6})_\infty}.
\end{align}
 \end{subequations}
\end{theorem}

The sum in \eqref{ntss} can be represented as
 \begin{multline*}(q^2;q^2)_\infty\oint\frac{(-q^9z^3;q^6)_\infty(1/z,q^2z;q^2)_\infty}{(-q^2z;q)_\infty(q^{3/2}z;-q)_\infty}\,\frac{dz}{2\pi\ti z}\\
 \begin{split}&=(q^2;q^2)_\infty\oint\frac{(-q^3\om z,-q^3\om^2 z,1/z,q^2z;q^2)_\infty}{(-q^2z,q^{3/2}z,-q^{5/2}z;q^2)_\infty}\,\frac{dz}{2\pi\ti z}\\
 &=(-q^2,q\om^2;q^2)_\infty\,{}_2\phi_1\left(\begin{matrix}
q^{1/2} \omega,-q^{3/2}\omega\\  -q^2
\end{matrix};q^2,{q}{\omega^2}\right),\end{split}
 \end{multline*}
  where we used
\eqref{tpia}.
It follows that Theorem \ref{ntst} is equivalent to the
following result.

\begin{theorem}\label{ntrt} We have the summation formula
\begin{multline}\label{kse}\,{}_2\phi_1\left(\begin{matrix}
q^{1/4} \omega,-q^{3/4}\omega\\  -q
\end{matrix};q,{q^{1/2}}{\omega^2}\right)\\
=\frac{(q^3;q^6)_\infty}{(q^{1/2}\omega^2;q)_\infty(-q^{3/4},q^{5/4};-q^{3/2})_{\infty}(q^2;q^6)_\infty}.
 \end{multline}
\end{theorem}

\subsection{Orthogonal polynomials}
\label{ops}

Our main tool to prove Theorem \ref{krta} and Theorem \ref{ntrt}
is quadratic transformations relating Askey--Wilson and Rogers polynomials.
From now on,  the variables $x$ and $z$ are related by
$$x=\frac{z+z^{-1}}2. $$
The Rogers (or continuous $q$-ultraspherical) polynomials \cite{ai,r3} are given by the generating function
\begin{equation}\label{cg}\frac{(atz,at/z;q)_\infty}{(tz,t/z;q)_\infty}=\sum_{n=0}^\infty C_n\left(x;a|q\right)t^n.
\end{equation}
The more general Askey--Wilson polynomials \cite{aw} have the generating function
\begin{equation}\label{awg}{}_2\phi_1\left(\begin{matrix}
a z,bz\\ ab
\end{matrix};q,\frac tz\right){}_2\phi_1\left(\begin{matrix}
c/z,d/z\\ cd
\end{matrix};q,tz\right)=\sum_{n=0}^\infty\frac{p_n(x;a,b,c,d|q)}{(q,ab,cd;q)_n}\,t^n,
\end{equation}
where the right-hand side converges for $|t|<\min(|z|,|z|^{-1})$  \cite{iw}.
 They are symmetric in the parameters $(a,b,c,d)$.

The Rogers polynomials appear as a special case of the Askey--Wilson polynomials in several different ways  \cite[\S 7.5]{gr}. For instance,
\begin{subequations}\label{awu}
\begin{align}\label{cpa}C_n(x;a^2|q)&=\frac{(a^4;q)_n}{(q,-a^2,a^2q^{1/2},-a^2q^{1/2};q)_n}\,p_n\big(x;a,-a,aq^{1/2},-aq^{1/2}|q\big),\\
\label{cpb}C_n(x;a^2|q^2)&=\frac{(a^2;q)_n}{(q^2,a^2q;q^2)_n}\,p_n\big(x;a,-a,q^{1/2},-q^{1/2}|q\big), \\
\label{cpc}C_{2n}(x;a|q)&=\frac{(a^2;q^2)_n}{(q,-a;q)_{2n}}\,p_n\big(2x^2-1;a,aq,-1,-q|q^2\big), \\
\label{cpd}C_{2n+1}(x;a|q)&=\frac{2(a^2;q^2)_{n+1}}{(q,-a;q)_{2n+1}}\,x\,p_n\big(2x^2-1;a,aq,-q,-q^2|q^2\big). \end{align}
\end{subequations}
These relations are easy to understand by comparing the orthogonality measures 
for the polynomials involved.

Using \eqref{awu} in \eqref{awg} leads to alternative generating functions for Rogers polynomials. 
In particular, \eqref{cpa} gives 
\begin{multline}\label{ncg}{}_2\phi_1\left(\begin{matrix}
a z,-az\\ -a^2
\end{matrix};q,\frac tz\right){}_2\phi_1\left(\begin{matrix}
aq^{1/2}/z,-aq^{1/2}/z\\ -a^2q
\end{matrix};q,tz\right)\\
=\sum_{n=0}^\infty\frac{(a^2q^{1/2},-a^2q^{1/2};q)_n}{(-a^2q,a^4;q)_n}\,C_n(x;a^2|q)t^n.
\end{multline}
The generating functions arising from \eqref{cpb}--\eqref{cpd} are as follows.

\begin{lemma}\label{acl}
The Rogers polynomials have the generating functions
\begin{subequations}\label{ag}
\begin{align}\label{acg}\frac{(tq/z;q^2)_\infty}{(tz;q^2)_\infty}\,
{}_2\phi_1\left(\begin{matrix}
a z,-az\\ -a^2
\end{matrix};q,\frac tz\right)
&=\sum_{n=0}^\infty\frac{(a^2q;q^2)_n}{(a^4;q^2)_n}\,C_n(x;a^2|q^2)t^n,\\
\label{adg}\frac{(-t/z;q)_\infty}{(tz;q)_\infty}\,
{}_2\phi_1\left(\begin{matrix}
a z^2,az^2q\\ a^2q
\end{matrix};q^2,\frac {t^2}{z^2}\right)
&=\sum_{n=0}^\infty\frac{(-a;q)_n}{(a^2;q)_n}\,C_n(x;a|q)t^n,
\end{align}
\end{subequations}
where the right-hand sides converge for $|t|<\min(|z|,|z|^{-1})$.
\end{lemma}

\begin{proof}
Choose
$(a,b,c,d)=(a,-a,q^{1/2},-q^{1/2})$ in \eqref{awg}. 
By \eqref{cpb}, the right-hand side can be identified with the 
right-hand side of \eqref{acg}.
On the left-hand side, we apply \eqref{aqbt} in the form
\begin{equation}\label{aqbts}{}_2\phi_1\left(\begin{matrix}
q^{1/2}/z,-q^{1/2}/z\\ -q
\end{matrix};q,tz\right)=\frac{(tq/z;q^2)_\infty}{(tz;q^2)_\infty}. 
\end{equation}
This proves \eqref{acg}.

Using \eqref{cpc}--\eqref{cpd} and \eqref{awg}, the right-hand side of \eqref{adg} can be written
\begin{multline*}
\sum_{n=0}^\infty\frac{t^{2n}}{(q,q^2,a^2q;q^2)_n}\,p_n\left(2x^2-1;a,aq,-1,-q|q^2\right)\\
+\frac{2tx}{1-q}\sum_{n=0}^\infty\frac{t^{2n}}{(q^2,q^3,a^2q;q^2)_n}\,p_n\left(2x^2-1;a,aq,-q,-q^2|q^2\right)\\
\begin{split}&={}_2\phi_1\left(\begin{matrix}
az^2,az^2q\\ a^2q
\end{matrix};q^2,\frac {t^2}{z^2}\right)\\
&\quad\times\left({}_2\phi_1\left(\begin{matrix}
-1/z^2,-q/z^2\\ q
\end{matrix};q^2,t^2z^2\right)+\frac{2tx}{1-q}\,{}_2\phi_1\left(\begin{matrix}
-q/z^2,-q^2/z^2\\ q^3
\end{matrix};q^2,t^2z^2\right)\right).\end{split}\end{multline*}
The final factor is evaluated by the $q$-binomial theorem in the form
\begin{multline*}
\sum_{n=0}^\infty \frac{(-1/z^2;q)_{2n}}{(q;q)_{2n}}\,(tz)^{2n}+\sum_{n=0}^\infty \frac{(-1/z^2;q)_{2n+1}}{(q;q)_{2n+1}}\,(tz)^{2n+1}\\
=\sum_{n=0}^\infty \frac{(-1/z^2;q)_{n}}{(q;q)_{n}}\,(tz)^{n}=\frac{(-t/z;q)_\infty}{(tz;q)_\infty}.
\end{multline*}
This proves \eqref{adg}.
\end{proof}

Applying \eqref{iht}
to the left-hand side of \eqref{acg} gives the equivalent identity
\begin{multline}\label{acge}\frac{(a^2t^2;q^2)_\infty}{(-a^2;q)_\infty(tz,t/z;q^2)_\infty}\,
{}_2\phi_2\left(\begin{matrix}
tz,t/z\\at,-at
\end{matrix};q, -a^2\right)\\
=\sum_{n=0}^\infty\frac{(a^2q;q^2)_n}{(a^4;q^2)_n}\,C_n(x;a^2|q^2)t^n.
\end{multline}
To prove \eqref{krx} we will need the following variation of \eqref{acge}.

\begin{lemma}\label{aclm}
The Rogers polynomials have the generating function
\begin{multline}\label{acgm}
\frac{(a^2t^2;q^2)_\infty}{(-a^2q^{-1};q)_\infty(tz,t/z;q^2)_\infty}\,
{}_2\phi_2\left(\begin{matrix}
tz,t/z\\at,-at
\end{matrix};q, -a^2q^{-1}\right)
\\
=\sum_{n=0}^\infty\frac{(a^2q^{-1};q^2)_n}{(a^4q^{-2};q^2)_n}\,C_n(x;a^2|q^2)t^n,
\end{multline}
where the right-hand side converges for $|t|<\min(|z|,|z|^{-1})$. 
\end{lemma}

\begin{proof}
We start from the identity
\begin{multline*}p_n(x;aq^{-1},-a,q^{1/2},-q^{1/2}|q)
=\frac{1-a^2q^{n-1}}{1-a^2q^{2n-1}}\,p_n(x;a,-a,q^{1/2},-q^{1/2}|q)\\
-{aq^{-1} (1-q^{2n})}\,p_{n-1}(x;a,-a,q^{1/2},-q^{1/2}|q), \end{multline*}
which is a special case of \cite[Eq.\ (7.6.8)]{gr} (in the first edition, a factor $(q;q)_n$ is missing on the right-hand side of (7.6.9)). Taking the even part in $a$ gives\begin{multline*}p_n(x;aq^{-1},-a,q^{1/2},-q^{1/2}|q)+p_n(x;a,-aq^{-1},q^{1/2},-q^{1/2}|q)\\
=2\frac{1-a^2q^{n-1}}{1-a^2q^{2n-1}}\,p_n(x;a,-a,q^{1/2},-q^{1/2}|q)
=2\frac{(q^2,a^2q^{-1};q^2)_n}{(a^2q^{-1};q)_n}\,C_n(x;a^2|q^2),
\end{multline*}
where we used \eqref{cpb}. It follows that the right-hand side of \eqref{acgm} equals
$$\frac 12\sum_{n=0}^\infty\frac{p_n(x;aq^{-1},-a,q^{1/2},-q^{1/2}|q)+p_n(x;a,-aq^{-1},q^{1/2},-q^{1/2}|q)}{(q,-a^2q^{-1},-q;q)_n}\,t^n. $$
Summing the series using \eqref{awg} and simplifying using \eqref{aqbts}
gives
$$\frac 12\frac{(tq/z;q^2)_\infty}{(tz;q^2)_\infty}\,\left(
{}_2\phi_1\left(\begin{matrix}
aq^{-1} z,-az\\ -a^2q^{-1}
\end{matrix};q,\frac tz\right) +{}_2\phi_1\left(\begin{matrix}
a z,-aq^{-1}z\\ -a^2q^{-1}
\end{matrix};q,\frac tz\right) \right).$$
By \eqref{iht}, this is equal to
\begin{multline*}\frac{(a^2t^2;q^2)_\infty}{2(tz,t/z;q^2)_\infty(-a^2q^{-1};q)_\infty}\left(
(1-atq^{-1})\,
{}_2\phi_2\left(\begin{matrix}
tz,t/z\\ atq^{-1},-at
\end{matrix};q,-a^2q^{-1}\right) \right.\\
\left.+
(1+atq^{-1})\,
{}_2\phi_2\left(\begin{matrix}
tz,t/z\\ at,-atq^{-1}
\end{matrix};q,-a^2q^{-1}\right) \right).
\end{multline*}
Adding the series termwise as in the proof of Lemma \ref{crl} gives the desired expression.
\end{proof}

Another key tool is the following identity for Rogers polynomials
at  $x=-1/2$ or, equivalently, 
 $z=\om$.

\begin{lemma} \label{scl}
One has
\begin{equation}\label{csv}C_n\left(-\frac 12;a\big|q\right)=\sum_{l=0}^{\lfloor n/3\rfloor}\frac{(a^3;q^3)_l(a^{-1};q)_{n-3l}}{(q^3;q^3)_l(q;q)_{n-3l}}\,a^{n-3l}.
 \end{equation}
\end{lemma}

To prove this, we simply use the $q$-binomial theorem to write the  generating function \eqref{cg} as
$$ \frac{(at\omega,at\omega^{-1};q)_\infty}{(t\omega,t\omega^{-1};q)_\infty}=\frac{(t;q)_\infty}{(at;q)_\infty}\cdot
\frac{(a^3t^3;q^3)_\infty}{(t^3;q^3)_\infty}
=\sum_{k=0}^\infty\frac{(a^{-1};q)_k}{(q;q)_k}\,(at)^k
\sum_{l=0}^\infty\frac{(a^3;q^3)_l}{(q^3;q^3)_l}\,t^{3l}.$$
In hypergeometric notation, \eqref{csv} is
$$C_n\left(-\frac 12;a\big|q\right)=\frac{a^n(a^{-1};q)_n}{(q;q)_n}\,
{}_4\phi_3\left(\begin{matrix}
q^{-n},q^{1-n},q^{2-n},a^3\\ aq^{1-n},aq^{2-n},aq^{3-n}
\end{matrix};q^3,q^3\right).$$
This is a balanced series, so 
one can obtain alternative expressions  by applying Sears' and Watson's transformation formulas to the right-hand side. 

Combining the results stated above we obtain the following sextic transformations, which seem to be new.

\begin{proposition}\label{ctp}
The transformation formulas 
\begin{subequations}
\begin{multline}\label{cta}{}_2\phi_1\left(\begin{matrix}
a \omega,-a\omega\\ -a^2
\end{matrix};q,a^2q^{-1}\om^2\right)\\
=\frac{(a^6;q^2)_\infty(a^6q^{-3};q^6)_\infty}{(a^2q^{-1}\omega^2,a^4q^{-1};q)_\infty}\,{}_3\phi_2\left(\begin{matrix}
a^2q,a^2q^3,a^2q^5\\ a^6q^2,a^6q^4
\end{matrix};q^6,a^6q^{-3}\right),
 \end{multline}
  \begin{multline}\label{ctb}{}_2\phi_1\left(\begin{matrix}
a \omega,aq\omega\\ a^2q
\end{matrix};q^2,a^2\om^2\right)\\
=\frac{(a^3;q)_\infty(-a^3;q^3)_\infty}{(a^2\omega^2,a^4;q^2)_\infty}\,{}_3\phi_2\left(\begin{matrix}
-a,-aq,-aq^2\\ a^3q,a^3q^2
\end{matrix};q^3,-a^3\right),\end{multline}
 \begin{multline}\label{ctam}{}_2\phi_2\left(\begin{matrix}
a^2q^{-1} \omega,a^2q^{-1}\om^2\\ a^3q^{-1},-a^3q^{-1}
\end{matrix};q,-a^2q^{-1}\right)\\
=\frac{(-a^2q^{-1};q)_\infty(a^6q^{-3};q^6)_\infty}{(a^4q^{-2};q)_\infty}\,{}_3\phi_2\left(\begin{matrix}
a^2q^{-1},a^2q,a^2q^3\\ a^6q^{-2},a^6q^2
\end{matrix};q^6,a^6q^{-3}\right),
 \end{multline}
 \begin{multline}\label{ctan}
{}_2\phi_2\left(\begin{matrix}
a^2q^{-3}\om ,a^2q^{-3}\om^2\\a^3q^{-3},-a^3q^{-3}
\end{matrix};q, -a^2q^{-1}\right)
\\
=
\frac{(-a^2q^{-1};q)_\infty(a^6q^{-9};q^6)_\infty}{(1-a^2q^{-3})(1-a^6q^{-6})(a^4q^{-3};q)_\infty}
\left(
{}_3\phi_2\left(\begin{matrix}
a^2q^{-1},a^2q,a^2q^3\\a^6q^{-4},a^6q^{-2}
\end{matrix};q^6, a^6q^{-9}\right)\right.\\
\left.-a^2q^{-3}\frac{1-a^2q^{-1}}{1-a^6q^{-4}}
\,{}_3\phi_2\left(\begin{matrix}
a^2q,a^2q^3,a^2q^5\\a^6q^{-2},a^6q^{2}
\end{matrix};q^6, a^6q^{-9}\right)\right),
\end{multline}
\end{subequations}
hold when the series on both sides converge.
\end{proposition}

\begin{proof}
We first consider \eqref{ctam}. We let
$z=\om$ in \eqref{acgm}, insert the first expression in Lemma \ref{scl}, replace $n$ by $n+3l$ and change the order of summation.
This gives
\begin{multline}\label{ctap}
\frac{(a^2t^2;q^2)_\infty}{(-a^2q^{-1};q)_\infty(t\om,t\om^2;q^2)_\infty}\,
{}_2\phi_2\left(\begin{matrix}
t\om ,t\om^2\\at,-at
\end{matrix};q, -a^2q^{-1}\right)
\\
\begin{split}&
=\sum_{l=0}^\infty \sum_{n=0}^\infty\frac{(a^2q^{-1};q^2)_{n+3l}(a^6;q^6)_l(a^{-2};q^2)_n}{(a^4q^{-2};q^2)_{n+3l}(q^6;q^6)_l(q^2;q^2)_n}\,a^{2n}t^{n+3l}\\
&=\sum_{l=0}^\infty \frac{(a^2q^{-1};q^2)_{3l}(a^6;q^6)_l}{(a^4q^{-2};q^2)_{3l}(q^6;q^6)_l}\,t^{3l}
{}_2\phi_1\left(\begin{matrix}
a^{-2},a^2q^{6l-1}\\ a^4q^{6l-2}
\end{matrix};q^2, a^2t\right).
\end{split}\end{multline}
We now specialize $t=a^2q^{-1}$ and apply the $q$-Gauss summation \eqref{qgs} in the form
$$ {}_2\phi_1\left(\begin{matrix}
a^{-2},a^2q^{6l-1}\\ a^4q^{6l-2}
\end{matrix};q^2,a^4q^{-1}\right)
=\frac{(a^6q^{-2},a^2q^{-1};q^2)_\infty}{(a^4q^{-2};q)_\infty}\frac{(a^4q^{-2};q^2)_{3l}}{(a^6q^{-2};q^2)_{3l}}.
$$
It is straight-forward to write the result as in \eqref{ctam}.

The identities \eqref{cta} and \eqref{ctb}
follow in exactly the same way from \eqref{acg} and \eqref{adg}, respectively.
We do not give the details.

Finally, we let $t=a^2q^{-3}$ 
in \eqref{ctap} and use \cite[(III.2)]{gr} to write
\begin{align*}
{}_2\phi_1\left(\begin{matrix}
a^{-2},a^2q^{6l-1}\\ a^4q^{6l-2}
\end{matrix};q^2, a^4q^{-3}\right)
&=\frac{(a^2q^{-1},a^6q^{6l-4};q^2)_\infty}{(a^4q^{-3},a^4q^{6l-2};q^2)_\infty}\,
{}_2\phi_1\left(\begin{matrix}
q^{-2},a^2q^{6l-1}\\ a^6q^{6l-4}
\end{matrix};q^2, a^2q^{-1}\right)\\
&=\frac{(a^2q^{-1},a^6q^{6l-4};q^2)_\infty}{(a^4q^{-3},a^4q^{6l-2};q^2)_\infty}\left(1-a^2q^{-3}\frac{1-a^2q^{6l-1}}{1-a^6q^{6l-4}}\right).
\end{align*}
Inserting this into \eqref{ctap} and simplifying gives \eqref{ctan}.

\end{proof}

\subsection{Proof of the remaining triple summations}
\label{tps}

We are now ready to prove
Theorem \ref{krta} and Theorem \ref{ntrt}, which imply the 
four Kanade--Russell conjectures
 \eqref{fsa}--\eqref{krx} and the new summation \eqref{ntss}.

\begin{proof}[Proof of \eqref{fsa}]
Let $a=q^{3/4}$ in  \eqref{cta}. Then, the right-hand side reduces to 
$$\frac{(q^{9/2};q^2)_\infty(q^{3/2};q^6)_\infty}{(q^{1/2}\om^2;q)_\infty(q^2;q)_\infty}\, {}_2\phi_1\left(\begin{matrix}
q^{5/2},q^{9/2}\\ q^{17/2}
\end{matrix};q^6,q^{3/2}\right).$$
Applying the $q$-Gauss summation \eqref{qgs} and simplifying gives \eqref{ksa}.
\end{proof}

\begin{proof}[Proof of \eqref{fsc}]
We  need  the generating function \eqref{ncg}.
Like in the proof of Proposition \ref{ctp}, we let  $z=\om$, insert the first expression in Lemma \ref{scl} and
change the order of summation. This gives
\begin{multline}\label{spg}
{}_2\phi_1\left(\begin{matrix}
a \om,-a\om\\ -a^2
\end{matrix};q,\frac t\om\right){}_2\phi_1\left(\begin{matrix}
aq^{1/2}\om^2,-aq^{1/2}\om^2\\ -a^2q
\end{matrix};q,t\om\right)\\
=\sum_{l=0}^\infty\frac{(a^2q^{1/2},-a^2q^{1/2};q)_{3l}(a^6;q^3)_l}{(-a^2q,a^4;q)_{3l}(q^3;q^3)_l}\,t^{3l}\\
\times\,{}_3\phi_2\left(\begin{matrix}
a^{-2},a^2q^{3l+1/2},-a^2q^{3l+1/2}\\ -a^2q^{3l+1},a^4q^{3l}
\end{matrix};q, a^2t\right).
\end{multline}
In the special case $a=q^{1/4}$ and $t=q^{1/2}$, the inner series reduces to 
$${}_2\phi_1\left(\begin{matrix}
q^{-1/2},-q^{3l+1}\\ -q^{3l+3/2}
\end{matrix};q, q\right)=\frac{(q^{1/2},-q^{3l+2};q)_\infty}{(q,-q^{3l+3/2};q)_\infty},$$
by  \eqref{qgs}.
The right-hand side of \eqref{spg} then simplifies to
\begin{multline*}
\frac{(q^{1/2},-q^{2};q)_\infty}{(q,-q^{3/2};q)_\infty}\sum_{l=0}^\infty \frac{(-q;q)_{3l}(q^{3/2};q^3)_l}{(-q^2;q)_{3l}(q^3;q^3)_l}\,q^{3l/2}\\
\begin{split}&=\frac{(q^{1/2},-q^{2};q)_\infty}{(q,-q^{3/2};q)_\infty}\,
{}_2\phi_1\left(\begin{matrix}
-q,q^{3/2}\\ -q^4
\end{matrix};q^3, q^{3/2}\right)\\
&=(1+q^{1/2})\frac{(q^{1/2},-q;q)_\infty(-q^{5/2},q^3;q^3)_\infty}{(q,-q^{1/2};q)_\infty(-q,q^{3/2};q^3)_\infty},
\end{split}\end{multline*}
again by \eqref{qgs}.
The left-hand side of \eqref{spg} is
$${}_2\phi_1\left(\begin{matrix}
q^{1/4} \om,-q^{1/4}\om\\ -q^{1/2}
\end{matrix};q,\frac {q^{1/2}}\om\right){}_2\phi_1\left(\begin{matrix}
q^{3/4}\om^2,-q^{3/4}\om^2\\ -q^{3/2}
\end{matrix};q,q^{1/2}\om\right).$$
The second factor is computed by \eqref{ksa}, with $\omega$ replaced by $\om^2$.
After  simplification, this results in \eqref{ksb}.
\end{proof}

\begin{proof}[Proof of \eqref{fsb}]
If we let $a=q^{3/4}$ in  \eqref{ctam}, the right-hand side becomes
$$\frac{(-q^{1/2};q)_\infty(q^{3/2};q^6)_\infty}{(q;q)_\infty}\,{}_2\phi_1\left(\begin{matrix}
q^{1/2},q^{9/2}\\ q^{13/2}
\end{matrix};q^6,q^{3/2}\right).$$
Applying the $q$-Gauss summation
 \eqref{qgs} and simplifying gives \eqref{ksc}.
\end{proof}

\begin{proof}[Proof of \eqref{krx}]
To prove the equivalent identity \eqref{ksd} we would like to let $a=q^{5/4}$ in \eqref{ctan}. 
Since the series on the right diverge at this point, 
we transform them using \cite[(III.9)]{gr}. This gives
 \begin{multline*}
{}_2\phi_2\left(\begin{matrix}
a^2q^{-3}\om ,a^2q^{-3}\om^2\\a^3q^{-3},-a^3q^{-3}
\end{matrix};q, -a^2q^{-1}\right)
=
\frac{(1-a^2q^{-2})(a^8q^{-8};q^6)_\infty}{(1-a^6q^{-6})(a^2q^{-3};q)_\infty(a^4q^{-3};q^2)_\infty}\\
\begin{split}&\times\left(
\frac{(a^4q^{-3};q^6)_\infty}{(a^6q^{-2};q^6)_\infty}\,
{}_3\phi_2\left(\begin{matrix}
a^2q,a^4q^{-7},a^4q^{-3}\\a^6q^{-4},a^8q^{-8}
\end{matrix};q^6, {a^4}{q^{-3}}\right)
\right.\\
&\quad\left.-a^2q^{-3}(1-a^2q^{-1})
\,
\frac{(a^4q;q^6)_\infty}{(a^6q^{-4};q^6)_\infty}\,
{}_3\phi_2\left(\begin{matrix}
a^2q,a^4q^{-7},a^4q^{-5}\\a^6q^{-2},a^8q^{-8}
\end{matrix};q^6, a^4q\right)
\right).
\end{split}\end{multline*}
By analytic continuation, this holds for $|a^4q^{-3}|<1$.
When $a=q^{5/4}$, the ${}_3\phi_2$-series on the right reduce to
$$
{}_1\phi_0\left(q^{-2};q^6, q^2\right)=0 $$
and
$$
{}_3\phi_2\left(\begin{matrix}
q^{7/2},q^{-2},1\\q^{11/2},q^2
\end{matrix};q^6,q^6\right)=1,
$$
respectively. After simplification we arrive at \eqref{ksd}.
\end{proof}

\begin{proof}[Proof of \eqref{ntss}]
Replace $q$ by $-q^{1/2}$ in \eqref{ctb} and then let $a=q^{1/4}$. This gives
\begin{multline*}{}_2\phi_1\left(\begin{matrix}
q^{1/4} \omega,-q^{3/4}\omega\\  -q
\end{matrix};q,{q^{1/2}}{\omega^2}\right)\\
=\frac{(q^{3/4};-q^{1/2})_\infty(-q^{3/4};-q^{3/2})_\infty}{(q^{1/2}\omega^2,q;q)_\infty}\,{}_2\phi_1\left(\begin{matrix}
-q^{1/4},q^{3/4}\\ q^{7/4}
\end{matrix};-q^{3/2},-q^{3/4}\right).
\end{multline*}
Applying  \eqref{qgs} and simplifying we obtain \eqref{kse}.
\end{proof}

\subsection{Quadratic transformations}

We conclude by discussing some  applications of Lemma~\ref{acl} to one-variable series.
We will need the generating function \cite{ims,rv}
$$\frac{(btz;q)_\infty}{(tz;q)_\infty} \, {}_3\phi_2\left(\begin{matrix}
a,b,az^2\\ a^2,btz
\end{matrix};q,\frac tz\right)\\
=\sum_{n=0}^\infty\frac{(b;q)_n}{(a^2;q)_n}\,C_n(x;a|q)t^n,\qquad x=\frac{z+z^{-1}}2.$$
Comparing it  with \eqref{ag} gives the quadratic transformation formulas
\begin{subequations}\label{qt}
\begin{align}\label{djt}
{}_2\phi_1\left(\begin{matrix}
a z,-az\\ -a^2
\end{matrix};q,\frac tz\right)
&=\frac{(a^2tzq;q^2)_\infty}{(tq/z;q^2)_\infty}\,
{}_3\phi_2\left(\begin{matrix}
a^2,a^2q,a^2z^2\\ a^4,a^2tz q
\end{matrix};q^2,\frac tz\right),\\
\label{jt}
{}_2\phi_1\left(\begin{matrix}
a z^2,az^2q\\ a^2q
\end{matrix};q^2,\frac {t^2}{z^2}\right)
&=\frac{(-atz;q)_\infty}{(-t/z;q)_\infty}\,
{}_3\phi_2\left(\begin{matrix}
a,-a,az^2\\ a^2,-atz 
\end{matrix};q,\frac tz\right).\end{align}
\end{subequations}
The identity \eqref{jt} is due to Jain \cite{j}.
Although we would be surprised if \eqref{djt} were new, 
we could not find it in the literature.

A pair of closely related identities  were obtained by Gessel and Stanton \cite[Eq.\ (5.6), (5.18)]{gs}, namely,
\begin{subequations}\label{gsi}
\begin{align}
{}_2\phi_1\left(\begin{matrix}
a,-a\\ -c
\end{matrix};q,cx\right)
&\doteq\frac{(a^2x;q^2)_\infty}{(x;q^2)_\infty}\,
{}_3\phi_2\left(\begin{matrix}
c,cq,a^2\\ c^2,q^2/x
\end{matrix};q^2,q^2\right),\\
{}_2\phi_1\left(\begin{matrix}
a,aq\\ c^2q
\end{matrix};q^2,c^2x^2\right)
&\doteq\frac{(ax;q)_\infty}{(x;q)_\infty}\,
{}_3\phi_2\left(\begin{matrix}
c,-c,a\\ c^2,q/x
\end{matrix};q,q\right).
\end{align}
\end{subequations}
Here, $\doteq$ means  identity as formal power series in $x$. 
Although the right-hand sides converge for generic values of the parameters, they have poles accumulating at $x=0$, so the corresponding power series diverge. Thus, \eqref{gsi} do  not hold as identities between convergent series (and the authors of \cite{gs} make no such claim). However, applying \cite[III.34]{gr} to \eqref{qt}
and changing parameters, one obtains
\begin{align*}
{}_2\phi_1\left(\begin{matrix}
a,-a\\ -c
\end{matrix};q,cx\right)
&=\frac{(a^2x;q^2)_\infty}{(x;q^2)_\infty}\,
{}_3\phi_2\left(\begin{matrix}
c,cq,a^2\\ c^2,q^2/x
\end{matrix};q^2,q^2\right)\\
&\quad+\frac{(a^2,c^2x;q^2)_\infty}{(-c,cx;q)_\infty(1/x;q^2)_\infty}
\,
{}_3\phi_2\left(\begin{matrix}
cx,cqx,a^2x\\ c^2x,q^2x
\end{matrix};q^2,q^2\right)
,\\
{}_2\phi_1\left(\begin{matrix}
a,aq\\ c^2q
\end{matrix};q^2,c^2x^2\right)
&=\frac{(ax;q)_\infty}{(x;q)_\infty}\,
{}_3\phi_2\left(\begin{matrix}
c,-c,a\\ c^2,q/x
\end{matrix};q,q\right)\\
&\quad+ \frac{(a,c^2x;q)_\infty}{(c^2q,c^2x^2;q^2)_\infty(1/x;q)_\infty}\,
{}_3\phi_2\left(\begin{matrix}
cx,-cx,ax\\ c^2x,qx
\end{matrix};q,q\right).
\end{align*}
These identities can be viewed as analytic versions of the formal identities \eqref{gsi}.

Finally,
we note that an interesting identity is obtained by substituting $z=\ti$ in \eqref{acg} and using \cite[Ex.\ (7.17)]{gr}
$$C_n(0;a|q)=\begin{cases}\displaystyle (-1)^k\frac{(a^2;q^2)_k}{(q^2;q^2)_k}, & n=2k,\\
0, & n=2k+1. \end{cases} $$
After replacing $a$ by $\ti a$ and $t$ by $\ti t$, we obtain 
\begin{equation}\label{nqt}{}_2\phi_1\left(\begin{matrix}
a,-a\\ a^2
\end{matrix};q,t\right)=\frac{(-t;q^2)_\infty}{(tq;q^2)_\infty}\,{}_2\phi_1\left(\begin{matrix}
-a^2q,-a^2q^3\\ a^4q^2
\end{matrix};q^4,t^2\right). \end{equation}
This quartic transformation can also be proved by combining the two identities \eqref{qt}. We have not been able to find it in the literature, but again find it hard to believe that it is new. It is interesting to compare \eqref{nqt} with Koornwinder's identities  \cite[Ex\ 3.2(iii)]{gr}
$$ {}_2\phi_1\left(\begin{matrix}
a,-a\\ a^2
\end{matrix};q,t\right)=(-t;q)_\infty\, {}_2\phi_1\left(\begin{matrix}
0,0\\ a^2q
\end{matrix};q^2,t^2\right)=\frac 1{(t;q)_\infty}\,{}_0\phi_1\left(-;a^2q
;q^2,a^2t^2q\right).$$


\begin{thebibliography}{99}
\bibitem[A]{a} G.\ E.\ 
Andrews, 
\emph{Schur's theorem, Capparelli's conjecture and q-trinomial coefficients}, in The Rademacher Legacy to Mathematics,
Amer.\ Math.\ Soc., 1994, pp.\  141–154.
\bibitem[AI]{ai} R.\ 
Askey and M.\ E.\ H.\ Ismail, \emph{A generalization of ultraspherical polynomials}, in Studies in Pure Mathematics, Birkh\"auser,  1983, pp.\ 55--78.
\bibitem[AW]{aw}
R.\ Askey and J.\ Wilson, \emph{Some basic hypergeometric orthogonal polynomials that generalize Jacobi polynomials}, Mem.\ Amer.\ Math.\ Soc.\ 54 (1985).
\bibitem[BJM]{bjm} K.\ Bringmann, C.\ Jennings-Shaffer and K.\ Mahlburg,
\emph{Proofs and reductions of various conjectured partition identities of Kanade and Russell}, J.\ Reine Angew. Math., to appear.
\bibitem[C]{c} 
S.\ Capparelli, \emph{On some representations of twisted affine Lie algebras and combinatorial identities}, J.\ Algebra 154 (1993), 335--355.
\bibitem[GR]{gr} G.\ Gasper and M.\ Rahman, Basic Hypergeometric Series,  $2^{\text{nd}}$ edition, Cambridge University Press, 2004. 
\bibitem[GS]{gs} I.\ Gessel and D.\ Stanton, \emph{Applications of $q$-Lagrange inversion to basic hypergeometric series}, Trans.\ Amer.\ Math.\ Soc.\ 277 (1983), 173--201.
\bibitem[IMS]{ims} M.\ E.\ H.\ Ismail, 
 D.\ R.\ Masson and S.\ K.\ Suslov, \emph{Some generating functions for $q$-polynomials}, in Special Functions, $q$-Series and Related Topics, 
 Amer. Math. Soc., 1997, pp.\ 91--108.
\bibitem[IW]{iw} M.\ E.\ H.\ Ismail and J.\ A.\ Wilson, 
\emph{Asymptotic and generating relations for the $q$-Jacobi and ${}_4\phi_3$ polynomials},
J.\ Approx.\ Theory 36 (1982), 43--54. 
\bibitem[J]{j} V.\ K.\ Jain, \emph{Some transformations of basic hypergeometric functions.\ II.}, SIAM J.\ Math.\ Anal.\ 12 (1981),  957--961.
\bibitem[KR]{kr} S.\ Kanade and M.\ C.\ Russell, \emph{Staircases to analytic sum-sides for many new integer partition identities of Rogers--Ramanujan type}, Electron.\ J.\ Combin.\ 26 (2019), 1.6. 
\bibitem[K]{k} K.\ Kur\c{s}ung\"oz,
\emph{Andrews--Gordon type series for Capparelli's and G\"ollnitz--Gordon identities},
J.\ Combin.\ Theory  A 165 (2019), 117--138. 
\bibitem[LM]{lm} 
J.\ Lepowsky and S.\ Milne, \emph{Lie algebraic approaches to classical partition identities}, Adv.\ Math.\ 29 (1978), 15--59.
\bibitem[LW1]{lw1} J.\ Lepowsky and R.\ L.\ Wilson, 
\emph{A Lie theoretic interpretation and proof of the Rogers--Ramanujan identities},
Adv.\ Math.\ 45 (1982), 21--72. 
\bibitem[LW2]{lw2}J.\ Lepowsky and R.\ L.\ Wilson, 
\emph{The structure of standard modules.\ I.\ Universal algebras and the Rogers--Ramanujan identities}, Invent.\ Math.\ 77 (1984), 199--290. 
\bibitem[M]{m} K.\ C.\ Misra, 
\emph{Specialized characters for affine Lie algebras and the Rogers--Ramanujan identities}, in Ramanujan Revisited, Academic Press, 1988, pp.\ 85--109.
 \bibitem[RV]{rv} M.\ Rahman and A.\ Verma, \emph{A $q$-integral representation of Rogers' $q$-ultraspherical polynomials and some applications}, Constr.\ Approx.\ 2 (1986), 1--10.
\bibitem[RR]{ra} S.\ Ramanujan and L.\ J.\ Rogers,
\emph{Proof of certain identities in combinatory analysis},
Proc.\ Cambridge Philos.\ Soc.\ 19 (1919), 214--216.
\bibitem[R1]{r2} L.\ J.\ Rogers, \emph{Second memoir on the expansion of certain infinite products}, Proc.\ Lond.\ Math.\ Soc.\ 25 (1894), 318--343. 
\bibitem[R2]{r3} L.\ J.\ Rogers, \emph{Third memoir on the expansion of certain infinite products}, 
 Proc.\ Lond.\ Math.\ Soc.\ 26 (1895), 15--32. 
\bibitem[Ro]{r} H.\ Rosengren, \emph{New proofs of determinant evaluations related to plane partitions}, Electronic J.\ Combin.\ 19 (2012) 15.
\bibitem[W]{w} 
G.\ N.\ Watson, 
\emph{Theorems stated by Ramanujan (VII): Theorems on continued fractions},
J.\ London Math.\ Soc.\ 4 (1929), 39--48. 
\vskip 3mm
\end{thebibliography}
 \end{document}